\documentclass[11pt, a4paper, twoside, leqno]{article}
\usepackage[T1]{fontenc}
\usepackage[utf8]{inputenc}
\usepackage{lmodern}
\usepackage[margin=27mm,headheight=14pt,headsep=7mm,footskip=11mm]{geometry}
\usepackage{amsmath,amssymb,amsthm,mathtools}
\usepackage{microtype,booktabs,array,longtable,enumitem}
\usepackage[dvipsnames]{xcolor}
\usepackage{fancyhdr}
\usepackage{hyperref}
\hypersetup{colorlinks=true,linkcolor=MidnightBlue,citecolor=MidnightBlue,
 urlcolor=MidnightBlue,pdftitle={A critical weighted Fourier estimate under vanishing spherical trace},
 pdfsubject={Research draft with cited inputs and a proof provenance concordance}}

\numberwithin{equation}{section}
\newtheorem{theorem}{Theorem}[section]
\newtheorem{proposition}[theorem]{Proposition}
\newtheorem{lemma}[theorem]{Lemma}
\newtheorem{corollary}[theorem]{Corollary}

\theoremstyle{remark}
\newtheorem{remark}[theorem]{Remark}

\newcommand{\Sn}{\mathbb S^{n-1}}
\newcommand{\cR}{\mathcal R}
\newcommand{\pc}{p_{\mathrm c}}
\newcommand{\R}{\mathbb{R}}
\newcommand{\C}{\mathbb{C}}
\newcommand{\Sch}{\mathcal{S}}
\newcommand{\F}{\mathcal{F}}
\newcommand{\Tr}{\mathcal{R}}
\newcommand{\Q}{\mathcal{Q}}
\newcommand{\dd}{\,\mathrm{d}}
\newcommand{\supp}{\operatorname{supp}}
\newcommand{\Rea}{\operatorname{Re}}

\newcommand{\norm}[1]{\lVert#1\rVert}
\newcommand{\abs}[1]{\lvert#1\rvert}
\newcommand{\ip}[2]{\langle#1,#2\rangle}

\newcolumntype{P}[1]{>{\raggedright\arraybackslash}p{#1}}

\newif\ifincludeconcordance
\includeconcordancetrue

\allowdisplaybreaks[1]
\setlist[enumerate]{itemsep=0.2em,topsep=0.4em}

\allowdisplaybreaks
\let\oldthanks\thanks
\renewcommand{\thanks}[1]{%
	\renewcommand{\thefootnote}{}
	\oldthanks{#1}
	\renewcommand{\thefootnote}{\arabic{footnote}}
}

\title {\textbf{ Endpoint weighted Fourier estimate with vanishing spherical trace }	
	\author{Wenjuan Li, \hspace{0.2cm}Manli Song, \hspace{0.2cm}Hongqiang Wang, \hspace{0.2cm} Huiju Wang}
	\hspace{0.2cm}
	\thanks{This work is supported by the National Key R\&D Program of China (No.2023YFA1010800); Natural Science Foundation of China (No.12301113; No.12271435); the Natural Science Basic Research Program of Shaanxi (Program No. 2026JC-YBMS-0013).}
    \thanks{Corresponding author: Manli Song}}

\begin{document}
\date{}
\maketitle
\thispagestyle{empty}
\begin{abstract}
We establish an endpoint weighted Fourier estimate under a vanishing
	restriction condition on the unit sphere. Let $n \geq 2$ and
	$p_c = 2(n+1)/(n+3)$. We prove that every
	$f \in L^{p_c}(\mathbb{R}^n)$ with zero spherical Fourier restriction
	satisfies
	$$
	\left\|
	\bigl|1-|\xi|^2\bigr|^{-1/2}\widehat{f}(\xi)
	\right\|_{L^2(\mathbb{R}^n)}
	\leq C_n \|f\|_{L^{p_c}(\mathbb{R}^n)},
	$$
	where the restriction condition is understood through the continuous
	Tomas--Stein restriction operator. This result gives a negative answer to Goldberg's
	counterexample-existence question and supplies a missing endpoint.
\end{abstract}

\section{Introduction and the main statement}
Let $n\geq2$ and $\alpha>0$. We consider the problem of determining
for which exponents $p$ and under which cancellation conditions on $f$
the weighted Fourier estimate
\begin{equation}\label{eq:general}
		\left\|
	\bigl|1-|\xi|^2\bigr|^{-\alpha}\widehat{f}(\xi)
	\right\|_{L^2(\mathbb{R}^n)}
	\leq C\|f\|_{L^p(\R^n)}
\end{equation}
holds. Here and throughout, the Fourier transform of $f$ is defined by
\begin{equation}\label{eq:fourier}
  \widehat f(\xi)=(2\pi)^{-n/2}\int_{\R^n}e^{-ix\cdot\xi}f(x)\,dx.
\end{equation}
The essential issue is the singularity along the unit sphere
$\Sn$. For $\alpha\geq\tfrac12$, the weight
$|1-|\xi|^2|^{-2\alpha}$ is not locally integrable across the sphere,
so \eqref{eq:general} cannot hold for arbitrary Schwartz functions.
This leads to the question of whether the natural cancellation condition
\begin{equation}\label{eq:vanishing}
	\widehat f\big|_{\Sn}=0
\end{equation}
is sufficient to recover a bounded operator. The corresponding localized
problem, with the integral restricted to a fixed neighborhood of $\Sn$,
isolates this question from the separate constraints at high frequencies.

One motivation comes from the solvability of constant-coefficient
differential equations whose symbols vanish on a hypersurface.
For the Helmholtz equation
\[
(-\Delta-1)u=f,
\]
the Fourier transform gives
$\widehat u(\xi)=(|\xi|^2-1)^{-1}\widehat f(\xi)$.
Thus \eqref{eq:general} with $\alpha=1$ is exactly the Fourier-side
estimate needed for an $L^2$ solution. The condition
\eqref{eq:vanishing} expresses compatibility of the data with the
characteristic sphere. In Agmon's weighted $L^2$ theory \cite{Agmon},
spatial decay supplies regularity of $\widehat f$ that can be combined
with a vanishing trace. The $L^p$ formulation asks whether the geometry
of the characteristic set and Fourier-side cancellation can instead
provide the required control through a translation-invariant norm.
This question is also related to the limiting absorption principle for
Schr\"odinger operators; see \cite{Goldberg,IonescuSchlag}.

A second motivation is the mapping theory of negative-order
Bochner--Riesz operators. With the convention
\[
\widehat{S^{-\alpha}f}(\xi)
=\mathbf 1_{\{|\xi|<1\}}(1-|\xi|^2)^{-\alpha}\widehat f(\xi),
\]
an estimate of the form \eqref{eq:general} immediately implies an
$L^p$--$L^2$ bound for $S^{-\alpha}$ on the class of admissible data.
For $0<\alpha<\tfrac12$, such bounds are available without
\eqref{eq:vanishing} in the range
$1\leq p\leq2(n+1)/(n+1+4\alpha)$; see \cite{BMO,Goldberg}.
At and above $\alpha=\tfrac12$, cancellation becomes indispensable.
The problem is therefore to understand how restricting the domain
of the multiplier to a Fourier restriction kernel changes its mapping
properties.

Goldberg \cite{Goldberg} developed this approach for the Helmholtz
equation and Bochner--Riesz multipliers. In particular, his
Theorem~3 establishes
\begin{equation}\label{eq:goldberg}
	\|S^{-\alpha}f\|_{L^2(\R^n)}
	\lesssim \|f\|_{L^p(\R^n)},
	\qquad \widehat f\big|_{\Sn}=0,
\end{equation}
for
\[
\frac12\leq\alpha<\frac32,\qquad
1\leq p\leq\frac{2(n+1)}{n+1+4\alpha},
\]
with the explicit exception
\begin{equation}\label{eq:endpoint}
	(\alpha,p)=\left(\frac12,\frac{2(n+1)}{n+3}\right).
\end{equation}
For $\tfrac12<\alpha<\tfrac32$, his Proposition~4 also controls
the two-sided singularity in a neighborhood of the sphere.
His Helmholtz result gives the global estimate at $\alpha=1$
in the range
\[
\max\left(1,\frac{2n}{n+4}\right)\leq p
\leq\frac{2(n+1)}{n+5},
\qquad n\geq3,\quad(n,p)\neq(4,1).
\]
These results show that the vanishing restriction can compensate for
a nonintegrable singularity, but leave open the critical transition
\eqref{eq:endpoint}.

Goldberg explicitly raised this endpoint in
\cite[Remark~3]{Goldberg}. In the subsequent problem collection
\cite[Problem~4]{Problems}, he asked whether there exists
$f\in L^{2(n+1)/(n+3)}(\R^n)$ such that
\[
\widehat f\big|_{\Sn}=0,
\qquad
|1-|\xi|^2|^{-1/2}\widehat f(\xi)\notin L^2(\R^n).
\]
The subtlety is that a vanishing restriction does not directly prescribe
a rate of decay toward the sphere. Writing
$A_f(\lambda)=\|\widehat f(\lambda\,\cdot)\|_{L^2(\Sn)}$, the endpoint requires
control of
\[
\int_{1-\varepsilon}^{1+\varepsilon}
\frac{A_f(\lambda)^2}{|\lambda-1|}\,d\lambda.
\]
Continuity of $A_f$ and the condition $A_f(1)=0$ alone do not guarantee
that this integral is finite. Moreover, the estimates with
$\alpha>\tfrac12$ apply at strictly smaller upper $p$-exponents and
do not directly yield the desired endpoint.

Our main result establishes this endpoint estimate in every dimension
$n\geq2$. Set
\[
\pc=\frac{2(n+1)}{n+3},
\qquad
X_n=\{f\in L^{p_c}:\mathcal{R}f:=\widehat{f}|_{\mathbb{S}^{n-1}}=0\}.
\]

\begin{theorem}\label{thm:main}
	Let $n\geq2$. There exists a constant $C_n$ such that every $f\in X_n$
	satisfies
	\begin{equation*}\label{eq:main}
		\left\|
	\bigl|1-|\xi|^2\bigr|^{-1/2}\widehat{f}(\xi)
	\right\|_{L^2(\mathbb{R}^n)}
		\leq C_n\|f\|_{L^{\pc}(\R^n)}.
	\end{equation*}
\end{theorem}

The integral is an ordinary nonnegative Lebesgue integral, with
$\widehat f$ understood as its Hausdorff--Young representative.
The theorem gives a negative answer to Goldberg's
counterexample-existence question and supplies the missing endpoint
in \eqref{eq:goldberg}.


\begin{remark} It is easy to get the characterization of the restriction kernel:
	\label{cor:kernel}
	\[
	f\in X_n
	\quad\Longleftrightarrow\quad
	|1-|\xi|^2|^{-1/2}\widehat f(\xi)\in L^2(\R^n).
	\]
Sufficiency is Theorem~\ref{thm:main}. The Stein--Tomas theorem, scaling, and approximation by Schwartz functions imply that $\lambda\mapsto\widehat f(\lambda\,\cdot)$ has a continuous $L^2(\Sn)$-valued representative for $\lambda>0$, agreeing with the polar slices of $\widehat f$ for almost every $\lambda$. If $\cR f\neq0$, its norm is bounded below by a positive constant for all $\lambda$ sufficiently close to $1$. Polar coordinates then give divergence of the weighted integral, since
$\int_{1-\varepsilon}^{1+\varepsilon}|\lambda-1|^{-1}\,d\lambda=\infty$.
\end{remark}

We next record two direct consequences of
Theorem~\ref{thm:main}. They describe, respectively, the endpoint Bochner--Riesz bounds, and
the dependence of fractional division estimates on the energy.

\begin{corollary}
	\label{cor:br}
	For every $f\in X_n$,
	\begin{equation*}\label{eq:br}
			\|S^{-\alpha}f\|_{L^2(\R^n)}
		\lesssim_n\|f\|_{L^{\pc}}.
	\end{equation*}
\end{corollary}

For $\lambda>0$, let $\cR_\lambda f=\widehat f(\lambda\,\cdot)|_{\mathbb{S}^{n-1}}$,
defined by continuous restriction, and introduce the fractional
division operator
\[
\widehat{T_\lambda f}(\xi)
=\bigl||\xi|^2-\lambda^2\bigr|^{-1/2}\widehat f(\xi).
\]
\begin{corollary}
	\label{cor:energy}
	For every $\lambda>0$ and every $f\in L^{\pc}(\R^n)$ satisfying
	$\cR_\lambda f=0$,
	\begin{equation*}\label{eq:energy}
		\|T_\lambda f\|_2
		\lesssim_n \lambda^{-1/(n+1)}\|f\|_{L^{\pc}}.
	\end{equation*}
	The implicit constant is independent of $\lambda$.
\end{corollary}

\begin{proof}
	Set $h(x)=f(x/\lambda)$. Then
	$\widehat h(\eta)=\lambda^n\widehat f(\lambda\eta)$,
	$\cR h=0$, and
	$\|h\|_{\pc}=\lambda^{n/\pc}\|f\|_{\pc}$.
	The change of variables $\xi=\lambda\eta$ and
	Theorem~\ref{thm:main} give
	\[
	\begin{split}
		\big\|\bigl||\xi|^2-\lambda^2\bigr|^{-1/2}\widehat f\big\|_2
		&=\lambda^{-n/2-1}
		\big\||1-|\eta|^2|^{-1/2}\widehat h(\eta)\big\|_2\\
		&\lesssim_n\lambda^{n/\pc-n/2-1}\|f\|_{\pc}
		=\lambda^{-1/(n+1)}\|f\|_{\pc}.
	\end{split}
	\]
	Plancherel's theorem completes the proof.
\end{proof}

\subsection{Organization of the Paper}
The remainder of this paper is organized as follows.
Section~2 collects the known results used in the proof, including
the Hardy--Littlewood--Sobolev inequality, localized dispersive
estimates, the Stein--Tomas restriction theorem, the scalar
logarithmic energy identity, and the zero-trace density theorem.
Section~3 establishes the auxiliary lemmas needed for the endpoint
argument: a local Fourier representation, a Hilbert-space-valued form of the logarithmic energy formula, 
and the cancellation of separable logarithmic terms.
Section~4 combines these ingredients to prove the local endpoint
estimate on graph patches and then on spherical patches.
Section~5 assembles the local estimates, controls the contribution
away from the sphere, and extends the Schwartz-class estimate to
the full zero-trace subspace of $L^{p_c}(\mathbb{R}^n)$ by the cited
density theorem and Fatou's lemma.

\subsection{AI usage}
We used ChatGPT 6.1 for English-language editing and proofreading, as well as to explore possible approaches to parts of the proofs. The principal ideas underlying the results were suggested by ChatGPT 6.1. The authors developed, carefully verified, and rewrote all the mathematical arguments, simplifying most of the proofs in the process. The authors are responsible for this paper.

\section{Preliminaries}\label{sec:known}


The following lemma is useful for our proof.

\begin{lemma}
[Logarithmic Hardy-Littlewood-Sobolev inequality] \label{lem:loghls}  Let $\beta\in(0,1)$ and $p=\frac{2}{2-\beta}$. Suppose $F\in L^p(\mathbb{R}^2)$ is nonnegative such that $\|F\|_{L^p(\mathbb{R}^2)}=1$. Then
\begin{equation*}\label{eq:loghls}
    \int_{\mathbb{R}}\int_{\mathbb{R}}|t-u|^{-\beta}F(t)F(u)|\log(|t-u|F(t)^{p/2}F(u)^{p/2})|dtdu\lesssim_\beta 1.
\end{equation*}
For points $(t,u)$ where $F(t)F(u)=0$, we adopt the convention $0\log 0=0$.
\end{lemma}
In order to prove Lemma \ref{lem:loghls}, we need the following lemma.
\begin{lemma}\label{logz}
    For any $z,\varepsilon>0$, we have
    \begin{equation*}
        |\log z|<\varepsilon^{-1}(z^\varepsilon+z^{-\varepsilon}).
        \end{equation*}
\end{lemma}
\begin{proof}
If $z\geq1$, then $z^\varepsilon\geq 1$. Note
    that
    \begin{equation} \label{logx-x}
      \log x<x, \quad \forall x>0.
    \end{equation} Taking $x=z^\varepsilon$ in \eqref{logx-x}, we get
    \begin{equation*}
        |\log z|=\log z<\varepsilon^{-1}z^\varepsilon.
    \end{equation*}
If $0<z<1$, then $0<z^\varepsilon<1$. Taking $x=z^{-\varepsilon}$ in \eqref{logx-x}, we have
    \begin{equation*}
        |\log z|=-\log z<\varepsilon^{-1}z^{-\varepsilon}.
    \end{equation*}
\end{proof}
\begin{proof}[Proof of Lemma \ref{lem:loghls}]
    It suffices to consider $(t,u)$ where $F(t)F(u)>0$ and $t\neq u$.
    Taking $\varepsilon=\frac{1}{2}\min(\beta,1-\beta)$ and $z=|t-u|F(t)^{p/2}F(u)^{p/2}$ in Lemma \ref{logz}, we have
    \begin{equation*}
      \Big|\log(|t-u|F(t)^{p/2}F(u)^{p/2})\Big|<\varepsilon^{-1}\left(|t-u|^\varepsilon F(t)^{p\varepsilon/2}F(u)^{p\varepsilon/2}+|t-u|^{-\varepsilon}F(t)^{-p\varepsilon/2}F(u)^{-p\varepsilon/2}\right),
    \end{equation*}
    which gives that
    \begin{align*}
     &\int_{\mathbb{R}}\int_{\mathbb{R}}|t-u|^{-\beta}F(t)F(u)|\log(|t-u|F(t)^{p/2}F(u)^{p/2})|dtdu\\
     &\lesssim_\beta \sum_\pm \int_{\mathbb{R}}\int_{\mathbb{R}}|t-u|^{-\lambda_\pm}F(t)^{a_\pm}F(u)^{a_\pm}dtdu,
    \end{align*}
    where $\lambda_\pm=\beta\mp\varepsilon$ and $a_\pm=1\pm\frac{p\varepsilon}{2}$.
    By virtue of $\varepsilon=\frac{1}{2}\min(\beta,1-\beta)$, it is easy to check that $\lambda_\pm \in(0,1)$. By the classical Hardy-Littlewood-Sobolev inequality, we have
    \begin{equation*}
      \int_{\mathbb{R}}\int_{\mathbb{R}}|t-u|^{-\lambda_\pm}F(t)^{a_\pm}F(u)^{a_\pm}dtdu\lesssim_\beta  \|F^{a_\pm}\|_{r_\pm}=\|F\|^{a_\pm}_{{a_\pm r_\pm}},
    \end{equation*}
    where $r_\pm=\frac{2}{2-\lambda_\pm}=\frac{2}{2-\beta\pm\varepsilon}$. Since $p=\frac{2}{2-\beta}$, we obtain $a_\pm r_\pm=r$. We complete the proof.
\end{proof}

We now recall some known useful lemmas.
\begin{lemma}[Localized curved-graph dispersive estimate]
\label{known:dispersion}
Let $n\geq2$ and $U\subset\R^{n-1}$ be a bounded open set, and let $\phi$ be
a real-valued smooth function on a neighborhood of $\overline U$.
Suppose that $D^2\phi$ is uniformly negative definite there.
Fix $\theta\in C_c^\infty(U)$ and a real-valued smooth bounded
extension $\widetilde\phi$ agreeing with $\phi$ near
$\operatorname{supp}\theta$.
Then, for every $1\leq p\leq2$,
\begin{equation*}\label{eq:dispersion}
 \left\|
 e^{iv\widetilde\phi(D_y)}\theta(D_y)
 \right\|_{L^p(\R^{n-1})\to L^{p'}(\R^{n-1})}
 \leq C_{\phi,\theta,p}
 (1+|v|)^{-(n-1)(1/p-1/2)}.
\end{equation*}
Here $D_y=-i\nabla_y$, and $p'$ denotes the conjugate exponent
of $p$.
\end{lemma}

Lemma \ref{known:dispersion} follows from the classical curved-surface Fourier decay
estimate of Littman~\cite{littman}, recalled in
Cuenin--Schippa~\cite{cueninschippa},
by interpolation with the $L^2$ multiplier bound.

\begin{lemma}[Stein--Tomas restriction theorem,  \cite{Tomas,stein}]
\label{known:restriction}
Let $n\geq2$ and $1\leq p\leq p_c$. For every
$u\in\Sch(\R^n)$,
\begin{equation*}\label{eq:restriction}
 \left\|\mathcal{R} u\right\|_{L^2(\mathbb{S}^{n-1})}
 \leq C_n\|u\|_{L^p(\R^n)}.
\end{equation*}
Consequently, Fourier restriction to $\mathbb{S}^{n-1}$ extends uniquely
to a bounded operator
\[
 \Tr:L^p(\R^n)\longrightarrow L^2(\mathbb{S}^{n-1}).
\]
\end{lemma}

\begin{lemma}[Scalar zero-mass logarithmic energy formula]
\label{known:scalarlog}
Let $q\in\Sch(\R;\C)$ satisfy
\[
 \int_\R q(t)\,\dd t=0.
\]
With the Fourier normalization in \eqref{eq:fourier},
\begin{equation}\label{eq:scalarlog}
 \int_\R\frac{|\widehat q(s)|^2}{|s|}\,\dd s
 =
 -\frac{1}{\pi}\operatorname{Re}
 \iint_{\R^2}
 \log|t-u|\,q(t)\overline{q(u)}\,\dd t\,\dd u.
\end{equation}
The double integral is absolutely convergent.
\end{lemma}

Lemma \ref{known:scalarlog} is the one-dimensional zero-mass case of
Frerick--M\"uller--Thomaser
\cite[Theorem~1 and Remark~3, p.~695]{fmt}.
The coefficient $1/\pi$ corresponds to the unitary Fourier
normalization in \eqref{eq:fourier}.

\begin{lemma}[Density in the spherical zero-trace subspace, \cite{stolyarov}]
\label{known:density}
Let $n\geq2$ and $1<p\leq p_c$. Then
\begin{equation*}\label{eq:density}
 \overline{
 \{u\in\Sch(\R^n):\Tr u=0\}
 }^{\,L^p(\R^n)}
 =
 \{f\in L^p(\R^n):\Tr f=0\}.
\end{equation*}
In particular, for every $f\in L^p(\R^n)$ satisfying
$\Tr f=0$, there exists a sequence $u_j\in\Sch(\R^n)$
such that $\Tr u_j=0$ for every $j$ and
\[
 \lim_{j\to \infty}\|u_j-f\|_{L^p(\R^n)}=0.
\]
\end{lemma}


\section{Deductions used in the endpoint argument}\label{sec:deductions}

\subsection{The local Fourier representation}
Let $n\geq2$. Fix a bounded convex open set $U\subset\R^{n-1}$ and a phase $\phi$
as in Lemma ~\ref{known:dispersion}.
Choose $K\subset U$ compact, $\theta\in C_c^\infty(U)$ with
$\theta=1$ on a neighborhood of $K$, and a real
$\widetilde\phi\in C_c^\infty(\R^{n-1})$ agreeing with $\phi$ near
$\supp\theta$. Define
\[
 U_\phi(v)=e^{iv\widetilde\phi(D_y)},\qquad H=L^2(\R^{n-1}),
\]
where $D_y=-i\nabla_y$. For $h\in\Sch(\R^{n-1}_y\times\R_t)$ with
$\supp\widehat h\subset K\times\R$, write $h_t(y)=h(y,t)$. Let $1\leq p\leq 2$ and set 
\begin{equation}\label{gandA}
 \begin{gathered}
 g(t)=U_\phi(-t)h_t,\qquad A(t)=\norm{h_t}_{L^p(\R^{n-1})},\\
 B(t,u)=\ip{g(t)}{g(u)}_H,\qquad \beta=(n-1)(\frac{1}{p}-\frac{1}{2}).
 \end{gathered}
\end{equation}

\begin{lemma}[Representation and slice bounds]\label{lem:representation}
In the preceding setting, $g\in\Sch(\R;H)$ and
\begin{equation*}\label{eq:representation}
 (\F_y\F_tg)(\eta,s)
 =\widehat h(\eta,\widetilde\phi(\eta)+s).
\end{equation*}
Moreover,
\begin{equation}\label{eq:bilinear}
 \abs{B(t,u)}\leq C\abs{t-u}^{-\beta}A(t)A(u)
 \qquad(t\ne u),
\end{equation}
and
\begin{equation}\label{eq:bandwidth}
 \norm{g(t)}_H=\norm{h_t}_{L^2(\mathbb{R}^{n-1})}\leq C A(t),
 \qquad \norm A_{L^p(\R)}=\norm h_{L^p(\R^{n})}.
\end{equation}
If $\widehat h(\eta,\widetilde\phi(\eta))=0$ for all $\eta$, then
\begin{equation}\label{eq:meanzero}
 \int_\R g(t)\dd t=0 \quad\text{in }H.
\end{equation}
\end{lemma}
\begin{proof}
The operators $U_\phi(v)$ are unitary on $H$.
The Schwartz bounds for $h$ imply rapid decay of
$\norm{\partial_t^j h_t}_H$ for every $j$.
Since $\widetilde\phi(D_y)$ is bounded on $H$,
\[
 \partial_t^m g(t)=U_\phi(-t)
 \sum_{j=0}^m\binom mj
       (-i\widetilde\phi(D_y))^{m-j}\partial_t^j h_t
\]
proves that $g\in\Sch(\R;H)$.
In particular, its Bochner integral and all required moments exist.

Since $\theta(D_y)h_u=h_u$, unitarity gives
\[
 B(t,u)=\ip{h_t}{U_\phi(t-u)\theta(D_y)h_u}_H.
\]
H\"older's inequality and Lemma \ref{known:dispersion}, yield
\[
 \abs{B(t,u)}\leq C(1+\abs{t-u})^{-\beta}A(t)A(u)
 \leq C\abs{t-u}^{-\beta}A(t)A(u).
\]
Then we obtain (\ref{eq:bilinear}).

For the first estimate in \eqref{eq:bandwidth}, write
$h_t=k_\theta*h_t$, where $k_\theta$ is the fixed Schwartz
convolution kernel of $\theta(D_y)$.
Young's inequality with $1/\rho=3/2-1/p$ gives
$\norm{h_t}_2\leq\norm{k_\theta}_\rho\norm{h_t}_p$.
Here $1<\rho<2$.
The last equality in \eqref{eq:bandwidth} follows from Tonelli's
theorem applied to $\abs{h(y,t)}^p$.

Finally, we prove (\ref{eq:meanzero}).  By the definition $g(t)=U_\phi(-t)h_t$, where $h_t(y)=h(y,t)$,
the Fourier transform in the $y$ variable satisfies
\[
 (\mathcal F_y g(t))(\eta)
 =
 e^{-it\widetilde\phi(\eta)}
 (\mathcal F_y h_t)(\eta).
\]
Taking the unitary Fourier transform in $t$, we obtain
\begin{align*}
 (\mathcal F_y\mathcal F_t g)(\eta,s)
 &=
 (2\pi)^{-1/2}
 \int_{\R}
 e^{-its}e^{-it\widetilde\phi(\eta)}
 (\mathcal F_y h_t)(\eta)\,\dd t \\
 &=
 (2\pi)^{-(d+1)/2}
 \int_{\R}\int_{\R^{n-1}}
 e^{-iy\cdot\eta}
 e^{-it(s+\widetilde\phi(\eta))}
 h(y,t)\,\dd y\,\dd t \\
 &=
 \widehat h\bigl(\eta,s+\widetilde\phi(\eta)\bigr).
\end{align*}
All interchanges of integrals are justified by
$h\in\Sch(\R^{n})$. Setting $s=0$ in the preceding identity gives
\[
 \mathcal F_y\bigl[(\mathcal F_tg)(0)\bigr](\eta)
 =
 \widehat h\bigl(\eta,\widetilde\phi(\eta)\bigr)
 =0.
\]
Since $H=L^2(\R^{n-1})$, Plancherel's theorem in $y$ yields
\[
 \|(\mathcal F_tg)(0)\|_H^2
 =
 \int_{\R^{n-1}}
 \left|
 \widehat h\bigl(\eta,\widetilde\phi(\eta)\bigr)
 \right|^2\,\dd\eta
 =0.
\]
Thus $(\mathcal F_tg)(0)=0$ in $H$. Finally,
$g\in\Sch(\R;H)\subset L^1(\R;H)$, so its Bochner integral
satisfies
\[
 \int_\R g(t)\,\dd t
 =
 (2\pi)^{1/2}(\mathcal F_tg)(0)
 =0
 \qquad\text{in }H.
\]

\end{proof}

\subsection{The Hilbert-valued form of the logarithmic energy formula}
The next lemma is a reduction to Lemma ~\ref{known:scalarlog}.
Its short proof checks the infinite-dimensional passage; it does not
rederive the scalar logarithmic kernel formula.

\begin{lemma}[Hilbert-valued reduction]\label{lem:hilbertlog}
Let $H$ be a separable complex Hilbert space and
$g\in\Sch(\R;H)$ with $\int_\mathbb{R} g(t)dt=0$. Set
$B(t,u)=\ip{g(t)}{g(u)}_H$. Then
\begin{equation*}\label{eq:hilbertlog}
 \int_\R\frac{\norm{\widehat g(s)}_H^2}{\abs s}\dd s
 =-\frac1\pi\Rea\iint_{\R^2}
       \log\abs{t-u}\,B(t,u)\dd t\dd u.
\end{equation*}
The double integral is absolutely convergent.
\end{lemma}

\begin{proof}
Set
\[
 b(t):=\|g(t)\|_H.
\]
Since $g\in\Sch(\R;H)$,
\[
 b\in L^1(\R)\cap L^2(\R),
 \qquad
 \int_\R |t|b(t)\,\dd t<\infty.
\]
Moreover, the Cauchy--Schwarz inequality in $H$ gives
\[
 |B(t,u)|
 =
 |\langle g(t),g(u)\rangle_H|
 \leq b(t)b(u).
\]
We estimate the logarithmic integral separately near and away
from the diagonal.

For $|t-u|<1$, the substitution $v=t-u$ and the
Cauchy--Schwarz inequality yield
\begin{align*}
 &\iint_{|t-u|<1}
 |\log|t-u||\,b(t)b(u)\,\dd t\,\dd u \\
 &\quad=
 \int_{|v|<1}|\log|v||
 \left(\int_\R b(u+v)b(u)\,\dd u\right)\dd v \\
 &\quad\leq
 \|b\|_{L^2(\R)}^2
 \int_{|v|<1}|\log|v||\,\dd v
 =
 2\|b\|_{L^2(\R)}^2.
\end{align*}
For $|t-u|\geq1$, we have
\[
 |\log|t-u||
 =\log|t-u|
 \leq |t-u|
 \leq |t|+|u|.
\]
Consequently,
\begin{align*}
 &\iint_{|t-u|\geq1}
 |\log|t-u||\,b(t)b(u)\,\dd t\,\dd u \\
 &\quad\leq
 \iint_{\R^2}
 (|t|+|u|)b(t)b(u)\,\dd t\,\dd u \\
 &\quad=
 2\|b\|_{L^1(\R)}
 \int_\R |t|b(t)\,\dd t
 <\infty.
\end{align*}
Combining these estimates with
$\iint b(t)b(u)\,\dd t\,\dd u=\|b\|_1^2$, we obtain
\begin{equation}\label{eq:logdomination}
 \iint_{\R^2}
 \bigl(1+|\log|t-u||\bigr)b(t)b(u)\,\dd t\,\dd u
 <\infty.
\end{equation}
In particular,
\begin{equation}\label{abscon}
 \iint_{\R^2}
 |\log|t-u||\,|B(t,u)|\,\dd t\,\dd u
 <\infty.
\end{equation}

Suppose first that $H$ is infinite-dimensional.
Let $\{e_k\}_{k\geq1}$ be an orthonormal basis of $H$, and define
\[
 g_k(t):=\langle g(t),e_k\rangle_H,
 \qquad
 P_Nx:=\sum_{k=1}^N\langle x,e_k\rangle_H e_k.
\]
Each $g_k$ is a scalar Schwartz function, and
\[
 \int_\R g_k(t)\,\dd t
 =
 \left\langle\int_\R g(t)\,\dd t,e_k\right\rangle_H
 =0.
\]
Applying \eqref{eq:scalarlog} to $g_k$ gives
\[
 \int_\R\frac{|\widehat{g_k}(s)|^2}{|s|}\,\dd s
 =
 -\frac1\pi\operatorname{Re}
 \iint_{\R^2}
 \log|t-u|\,g_k(t)\overline{g_k(u)}\,\dd t\,\dd u.
\]
Since the inner product is linear in its first argument,
\[
 \widehat{g_k}(s)=\langle\widehat g(s),e_k\rangle_H.
\]
Thus
\[
 \sum_{k=1}^N|\widehat{g_k}(s)|^2
 =\|P_N\widehat g(s)\|_H^2
\]
and
\[
 \sum_{k=1}^N g_k(t)\overline{g_k(u)}
 =\langle P_Ng(t),P_Ng(u)\rangle_H.
\]
Summing the scalar identities therefore yields
\[
 \begin{aligned}
 \int_\R\frac{\|P_N\widehat g(s)\|_H^2}{|s|}\,\dd s
 &=
 -\frac1\pi\operatorname{Re}
 \iint_{\R^2}\log|t-u|\langle P_Ng(t),P_Ng(u)\rangle_H\,\dd t\,\dd u.
 \end{aligned}
\]

By Parseval's identity in $H$,
\[
 \|P_N\widehat g(s)\|_H^2
 =
 \sum_{k=1}^N|\langle\widehat g(s),e_k\rangle_H|^2
 \uparrow
 \|\widehat g(s)\|_H^2.
\]
The monotone convergence theorem therefore gives
\[
 \lim_{N\to\infty}
 \int_\R\frac{\|P_N\widehat g(s)\|_H^2}{|s|}\,\dd s
 =
 \int_\R\frac{\|\widehat g(s)\|_H^2}{|s|}\,\dd s.
\]
On the other hand, $P_Ng(t)\to g(t)$ in $H$, so
\[
 \langle P_Ng(t),P_Ng(u)\rangle_H
 \longrightarrow B(t,u).
\]
Since orthogonal projections are contractions,
\[
 \begin{aligned}
  \left|
 \log|t-u|\,
 \langle P_Ng(t),P_Ng(u)\rangle_H
 \right|
 & \leq
 |\log|t-u||\,\|g(t)\|_H\|g(u)\|_H \\
& =
 |\log|t-u||\,b(t)b(u).
 \end{aligned}
\]
The last expression is integrable by
\eqref{eq:logdomination}. Dominated convergence therefore yields
\[
 \begin{aligned}
 &\lim_{N\to\infty}
 \iint_{\R^2}
 \log|t-u|\,
 \langle P_Ng(t),P_Ng(u)\rangle_H\,\dd t\,\dd u
 \iint_{\R^2}\log|t-u|\,B(t,u)\,\dd t\,\dd u.
 \end{aligned}
\]
Passing to the limit in the finite-dimensional identity proves
\[
 \int_\R\frac{\|\widehat g(s)\|_H^2}{|s|}\,\dd s
 =
 -\frac1\pi\operatorname{Re}
 \iint_{\R^2}\log|t-u|\,B(t,u)\,\dd t\,\dd u.
\]
The right-hand side is finite by (\ref{abscon}), so the integral on
the left is finite as well.
\end{proof}

\subsection{The data-dependent separable logarithmic terms}
The elementary cancellation principle is that adding $c(t)+c(u)$
to a kernel does not change its integral against $B(t,u)$ when
$\int_\mathbb{R} g(t)dt=0$ and the additional terms are integrable.
The particular choice of $c$ below connects this principle with
Lemma~\ref{lem:loghls}. We check its integrability explicitly.

\begin{lemma}[Legitimate logarithmic renormalization]\label{lem:cancellation}
Use the notation \eqref{gandA}, suppose $\int_\mathbb{R} g(t)dt=0$, and assume
$M=\norm h_{L^p(\mathbb{R}^n)}>0$. Define
\[
 F(t)=\frac{A(t)}{M},\qquad
 c(t)=\begin{cases}
       (p/2)\log F(t),&F(t)>0,\\
       0,&F(t)=0.
      \end{cases}
\]
The two functions $c(t)B(t,u)$ and $c(u)B(t,u)$ are absolutely
integrable on $\R^2$, and their integrals are zero. Consequently,
\begin{equation}\label{eq:renormalization}
 \begin{split}
 &\iint_{\R^2}\log\abs{t-u}\,B(t,u)\dd t\dd u =\iint_{\R^2}
       \log\!\left(\abs{t-u}F(t)^{p/2}F(u)^{p/2}\right)
       B(t,u)\dd t\dd u.
 \end{split}
\end{equation}
The entire integrand on the second line is defined to be zero
where $F(t)F(u)=0$, and is absolutely integrable.
Also $\norm F_{L^p(\mathbb{R})}=1$.
\end{lemma}
\begin{proof}
Since $h$ is Schwartz, for each $N$ there is a constant $C_N$ with
$A(t)\leq C_N(1+\abs t)^{-N}$.
For example, first bound
$\abs{h(y,t)}\leq C_{N,L}(1+\abs t)^{-N}(1+\abs y)^{-L}$
with $Lp>d$, and then integrate in $y$.
Thus $F$ is bounded and $F^\delta\in L^1(\R)$ for every $\delta>0$.
In particular,
\begin{equation}\label{eq:entropy}
 \int_\R F(t)\abs{\log F(t)}\dd t<\infty,
\end{equation}
where the integrand is zero at $F=0$.
Indeed, for $0<F\leq1$, we use $F\abs{\log F}\leq C F^{1/2}$;
for $F>1$, we use boundedness of $F$ and integrability of $F$.

By \eqref{eq:bandwidth} and \eqref{eq:entropy},
\[
 \int_\R\norm{c(t)g(t)}_H\dd t
 \leq CM\int_\R F(t)\abs{\log F(t)}\dd t<\infty.
\]
Also $g\in L^1(\R;H)$. It follows that
\[
\iint_{\R^2}\abs{c(t)B(t,u)}\dd t\dd u
 \leq\left(\int_{\R}\norm{c(t)g(t)}_H\dd t\right)
       \left(\int_{\R}\norm{g(u)}_H\dd u\right)<\infty,
\]
and likewise for $c(u)B(t,u)$.
Fubini's theorem and $\int_\mathbb{R} g(t)dt=0$ now give
\begin{align*}
 \iint_{\R^2} c(t)B(t,u)\dd t\dd u
 &=\ip{\int_{\R} c(t)g(t)\dd t}{\int_{\R} g(u)\dd u}_H=0,\\
 \iint_{\R^2} c(u)B(t,u)\dd t\dd u
 &=\ip{\int_{\R} g(t)\dd t}{\int_{\R} c(u)g(u)\dd u}_H=0.
\end{align*}
The function $c$ is real-valued, as required for the second equality
under our inner product convention.

For $t\ne u$ and $F(t)F(u)>0$,
\[
 \log\!\left(\abs{t-u}F(t)^{p/2}F(u)^{p/2}\right)
 =\log\abs{t-u}+c(t)+c(u).
\]
If $F(t)=0$, then $A(t)=0$, so $h_t=0$ and $g(t)=0$ in $H$.
The same holds with $u$ in place of $t$.
Thus $B(t,u)=0$ whenever $F(t)F(u)=0$.
The convention for the entire integrand is therefore consistent
with this expansion after multiplication by $B$.
The bare logarithmic term is absolutely integrable by
Lemma~\ref{lem:hilbertlog}, and the other two terms were just checked.
Integrating the expansion proves \eqref{eq:renormalization}.

Finally, \eqref{eq:bandwidth} and Tonelli give
$\int_{\mathbb{R}} F(t)^p\dd t=M^{-p}\int_{\mathbb{R}} A(t)^p\dd t=1$.
This normalization is not needed to derive
\eqref{eq:renormalization}; it is used later in
Lemma~\ref{lem:loghls}.
The qualitative integrability bounds in this proof may depend on
$h$. They are not the uniform constants in the endpoint estimate.
\end{proof}

\section{The local endpoint estimate}\label{sec:local}
Now we are ready to give the estimate on a
fixed graph. The integration in the normal frequency is two-sided.

\begin{proposition}[Local graph estimate]\label{prop:local}
Let the phase, cutoff, and compact set be as in the local setup
of Section~\ref{sec:deductions}. Suppose
$h\in\Sch(\R^n)$, $\supp\widehat h\subset K\times\R$, and
$\widehat h(\eta,\widetilde\phi(\eta))=0$ for every $\eta$.
Then
\begin{equation*}\label{eq:graphenergy}
 \Q_\phi(h):=\int_{\R^{n-1}}\int_\R
 \frac{\abs{\widehat h(\eta,\widetilde\phi(\eta)+s)}^2}{\abs s}
 \dd s\dd\eta
 \leq C_{\phi,\theta,n}\norm h_{L^{p_c}(\mathbb{R}^n)}^2.
\end{equation*}
\end{proposition}
\begin{proof}
Let $g,B$ be defined by \eqref{gandA} and $A(t)=\norm{h_t}_{L^{p_c}(\R^{n-1})}$. Denote $\beta_c=(n-1)(\frac{1}{p_c}-\frac{1}{2})=\frac{n-1}{n+1}$ and then $p_c=\frac{2}{2-\beta_c}$.
Lemma~\ref{lem:representation} gives $g\in\Sch(\R;H)$ and
$\int_\mathbb{R} g(t)dt=0$. Transverse Plancherel and Tonelli's theorem give
\begin{equation*}\label{eq:energyrepresentation}
 \Q_\phi(h)=\int_\R
       \frac{\norm{(\F_tg)(s)}_H^2}{\abs s}\dd s.
\end{equation*}
This is an $L^2$ identity, not a statement about $L^{p_c}$ boundedness
of a nonlinear change of frequency variables.

If $\norm h_{L^{p_c}(\mathbb{R}^n)}=0$, there is nothing to prove.
Otherwise put $M=\norm h_{L^{p_c}(\mathbb{R}^n)}$ and $F=A/M$.
First apply Lemma~\ref{lem:hilbertlog}; then use the vanishing of
the two separable terms established in
Lemma~\ref{lem:cancellation}. This gives the exact identity
\begin{equation*}\label{eq:normalizedenergy}
 \begin{split}
 \Q_\phi(h)&=-\frac1\pi\Rea\iint_{\R^2}B(t,u)
       \log\!\left(\abs{t-u}F(t)^{p_c/2}F(u)^{p_c/2}\right)
       \dd t\dd u.
 \end{split}
\end{equation*}
The double integral is already known to be absolutely convergent.
Where $F(t)F(u)=0$, its entire integrand is assigned the value zero,
as in Lemma~\ref{lem:cancellation}; values on the diagonal can also
be assigned arbitrarily.

Using \eqref{eq:bilinear}, $A(t)A(u)=M^2F(t)F(u)$, and the
nonnegativity of $\Q_\phi(h)$, we obtain
\begin{align*}
 \Q_\phi(h)
 &\leq\frac1\pi\iint_{\mathbb{R}^2}
 \left|\log\!\left(\abs{t-u}F(t)^{p_c/2}F(u)^{p_c/2}\right)\right|
 \abs{B(t,u)}\dd t\dd u\\
 &\leq CM^2\iint_{\mathbb{R}^2}\abs{t-u}^{-\beta_c}F(t)F(u)
 \left|\log\!\left(\abs{t-u}F(t)^{p_c/2}F(u)^{p_c/2}\right)\right|
 \dd t\dd u\\
 &\leq CM^2.
\end{align*}
For the last inequality we use $\norm F_{L^{p_c}(\mathbb{R})}=1$ and
Lemma~\ref{lem:loghls}, with $p_c=2/(2-\beta_c)$.
This is the step where normalization of the slice norms is needed.
All constants in this final estimate depend only on the fixed
phase, cutoff, and dimension.
\end{proof}

\subsection{Passing from a graph weight to the spherical weight}
Consider a fixed small neighborhood of the north pole and write
$\xi=(\eta,\lambda)\in\mathbb{R}^{n-1}\times\mathbb{R}$, with
\[
 K\subset\{\eta\in\mathbb{R}^{n-1}:  \abs\eta\leq r\},\quad r<1,\qquad
 \phi(\eta)=\sqrt{1-\abs\eta^2}.
\]
On a slightly larger ball still contained in $\{\eta\in\mathbb{R}^{n-1}: \abs\eta<1\}$,
\[
 D^2\phi(\eta)=-\frac{I}{\phi(\eta)}
                -\frac{\eta\eta^{\mathsf T}}{\phi(\eta)^3}.
\]
Thus the phase satisfies the fixed nondegenerate curvature hypotheses.
The cutoffs in the local setup are chosen inside that ball.

\begin{corollary}[Spherical patch estimate]\label{cor:patch}
Suppose $h\in\Sch(\R^n)$ satisfies $\Tr h=0$ and
\[
 \supp\widehat h\subset K\times(1/2,3/2).
\]
Then $\Q(h)\leq C\norm h_{L^{p_c}(\mathbb{R}^n)}^2$, with a constant for the fixed patch.
\end{corollary}
\begin{proof}
On the Fourier support of $h$,
\begin{equation*}\label{eq:spherefactor}
 \abs{1-\abs\xi^2}
 =\abs{\lambda^2-\phi(\eta)^2}
 =\abs{\lambda-\phi(\eta)}\abs{\lambda+\phi(\eta)}.
\end{equation*}
The second factor is bounded below by
\[
 \lambda+\phi(\eta)\geq\frac12+\sqrt{1-r^2}=:c_*>0.
\]
Consequently,
\begin{equation*}\label{eq:weightcomparison}
 \Q(h)\leq c_*^{-1}
       \int_K\int_\R
       \frac{\abs{\widehat h(\eta,\lambda)}^2}
            {\abs{\lambda-\phi(\eta)}}\dd\lambda\dd\eta.
\end{equation*}
The integral can be taken over the whole $\lambda$ axis because
$\widehat h$ is zero outside the prescribed support.

For each fixed $\eta$, the translation
$s=\lambda-\phi(\eta)$ has $\dd\lambda=\dd s$. Hence
\begin{align*}
 \int_K\int_\R
 \frac{\abs{\widehat h(\eta,\lambda)}^2}
      {\abs{\lambda-\phi(\eta)}}\dd\lambda\dd\eta
 &=\int_K\int_\R
 \frac{\abs{\widehat h(\eta,\phi(\eta)+s)}^2}{\abs s}\dd s\dd\eta\\
 &=\Q_\phi(h).
\end{align*}
The last equality uses $\widetilde\phi=\phi$ on $K$ and the
vanishing of $\widehat h$ for $\eta\notin K$.
These substitutions are valid for nonnegative extended integrals;
finiteness need not be assumed in advance.

To apply Proposition~\ref{prop:local}, it remains to check its
vanishing assumption. For $\eta\in K$, the point
$(\eta,\phi(\eta))$ lies on $\mathbb{S}^{n-1}$, so
$\widehat h(\eta,\widetilde\phi(\eta))=0$.
Indeed, the spherical trace of a Schwartz function is smooth,
and vanishing in $L^2(\mathbb{S}^{n-1})$ implies pointwise vanishing.
For $\eta\notin K$, the same graph vanishing follows from the
Fourier support condition. Combining the weight comparison,
translation, and local graph estimate gives
\begin{equation*}\label{eq:spherelocal}
 \Q(h)\leq c_*^{-1}\Q_\phi(h)\leq C\norm h_{L^{p_c}(\mathbb{R}^n)}^2.
\end{equation*}
Only the energy integral has been rewritten. The $L^{p_c}$ norm being
estimated is the original norm of $h$.
\end{proof}

\section{Globalization and general zero-trace data}\label{sec:global}
\subsection{Fixed localization and the contribution away from the sphere}
Choose finitely many ambient neighborhoods of $\mathbb{S}^{n-1}$, each of
which becomes a patch of the preceding type after a fixed rotation.
Choose a smooth partition
\begin{equation}\label{eq:partition}
 \chi_0+\sum_{j=1}^J\chi_j=1,
\end{equation}
where $\chi_j\in C_c^\infty(\mathbb{R}^n)$ is supported in the $j$th neighborhood
for $j\geq1$, and $\chi_0$ vanishes on a tubular neighborhood of
the sphere and equals one for sufficiently large $\abs\xi$.
All these choices are made once, depending only on $n$.

For completeness, we construct such a partition by selecting nonnegative bump functions $\rho_j$ in these neighborhoods, satisfying
$\rho=\sum\limits_{j=0}^J\rho_j>0$ near $\mathbb{S}^{n-1}$. We then choose \(\kappa\in C_c^\infty(\mathbb R^n)\) whose support is compactly contained in \(\{x\in\mathbb{R}^n: \rho(x)>0\}\) and which equals 1 on a smaller tubular neighborhood of the sphere.
 Set
$\chi_j=\kappa\rho_j/\rho$,
and $\chi_0=1-\kappa$.
For $j\geq1$, $\chi_j(D)$ is convolution with the fixed kernel
\[
 k_j(x)=(2\pi)^{-n/2}\int_{\R^n}e^{ix\cdot\xi}\chi_j(\xi)\dd\xi
 \in\Sch(\R^n),
\]
so Young's inequality gives
\begin{equation}\label{eq:multiplier}
 \norm{\chi_j(D)f}_{L^p(\mathbb{R}^n)}\leq\norm{k_j}_{L^1(\mathbb{R}^n)}\norm f_{L^p(\mathbb{R}^n)},\quad \forall p\geq1.
\end{equation}

\begin{lemma}[Elementary estimate away from the sphere]\label{lem:away}
For every $f\in L^{p_c}(\R^n)$,
\begin{equation*}\label{eq:away}
 \int_{\R^n}
 \frac{\abs{\chi_0(\xi)\widehat f(\xi)}^2}
      {\abs{1-\abs\xi^2}}\dd\xi
 \leq C_n\norm f_{L^{p_c}(\mathbb{R}^n)}^2.
\end{equation*}
\end{lemma}
\begin{proof}
Let $W(\xi)=\abs{\chi_0(\xi)}^2/\abs{1-\abs\xi^2}$,
assigning the value zero on the sphere.
The function $W$ is bounded on bounded sets and is
$O(\abs\xi^{-2})$ at infinity. For $q=(n+1)/2>1$,
\[
 \int_{\abs\xi\geq2}\abs\xi^{-2q}\dd\xi
 =\abs{\mathbb{S}^{n-1}}\int_2^\infty r^{n-1-2q}\dd r
 =\abs{\mathbb{S}^{n-1}}\int_2^\infty r^{-2}\dd r<\infty.
\]
Thus $W\in L^q(\R^n)$.
Since $q'=(n+1)/(n-1)$ and $2q'=p_c'$, H\"older and
Hausdorff-Young inequality give
\[
 \int_{\mathbb{R}^n} W(\xi)\abs{\widehat f(\xi)}^2d\xi
 \leq\norm W_{L^q(\mathbb{R}^n)}\norm{\widehat f}_{L^{2q'}(\mathbb{R}^n)}^2
 =\norm W_{L^q(\mathbb{R}^n)}\norm{\widehat f}_{L^{p_c'}(\mathbb{R}^n)}^2
 \leq C_n\norm f_{L^{p_c}(\mathbb{R}^n)}^2.
\]
\end{proof}

\subsection{The estimate on the Schwartz zero-trace class}
\begin{proposition}\label{prop:schwartz}
For $f\in\Sch(\R^n)$ with $\Tr f=0$,
\begin{equation*}\label{eq:schwartz}
 \Q(f)\leq C_n\norm f_{L^{p_c}(\mathbb{R}^n)}^2.
\end{equation*}
\end{proposition}
\begin{proof}
Set $h_j=\chi_j(D)f$ for $0\leq j\leq J$.
For $j\geq1$, $h_j$ is Schwartz and its spherical trace vanishes:
\[
 \widehat h_j|_{\mathbb{S}^{n-1}}=\chi_j|_{\mathbb{S}^{n-1}}\widehat f|_{\mathbb{S}^{n-1}}=0.
\]
Fixed rotations preserve the $L^{p_c}$ norm, the spherical weight,
and the Fourier normalization. Corollary~\ref{cor:patch} and
\eqref{eq:multiplier} therefore give
$\Q(h_j)\leq C_j\norm{k_j}_{L^1(\mathbb{R}^n)}^2\norm f_{L^{p_c}(\mathbb{R}^n)}^2$ for $j\geq1$.
For $j=0$, Lemma~\ref{lem:away} gives
$\Q(h_0)\leq C_n\norm f_{L^{p_c}(\mathbb{R}^n)}^2$ directly.

By \eqref{eq:partition},
$\widehat f=\sum\limits_{j=0}^J\widehat h_j$.
The finite-dimensional Cauchy--Schwarz inequality gives
\[
 \abs{\widehat f(\xi)}^2
 \leq(J+1)\sum_{j=0}^J\abs{\widehat h_j(\xi)}^2.
\]
Multiplying by the nonnegative weight and integrating yields
$\Q(f)\leq(J+1)\sum\limits_{j=0}^J\Q(h_j)\leq C_n\norm f_{L^{p_c}(\mathbb{R}^n)}^2$.
The finite collection of constants is fixed in advance.
\end{proof}

\subsection{The density theorem and the final limit}
The passage from Schwartz functions to general $L^{p_c}$ functions
uses Lemma~\ref{known:density}; its proof is not repeated.
Ordinary Schwartz approximation alone would not suffice here,
because it need not preserve exact zero trace.

\begin{proof}[Proof of Theorem~\ref{thm:main}]
Let $f\in L^{p_c}(\R^n)$ satisfy $\Tr f=0$.
By Lemma \ref{known:density}, choose $f_j\in\Sch(\R^n)$
with $\Tr f_j=0$ and $\lim\limits_{j\to\infty}\norm{f_j-f}_{L^{p_c}(\mathbb{R}^n)}=0$.
Proposition~\ref{prop:schwartz} gives the uniform bound
\begin{equation}\label{eq:approximatebound}
 \Q(f_j)\leq C_n\norm{f_j}_{L^{p_c}(\mathbb{R}^n)}^2.
\end{equation}
Hausdorff--Young gives
$\norm{\widehat f_j-\widehat f}_{L^{p_c'}(\mathbb{R}^n)}\leq C_n\norm{f_j-f}_{L^{p_c}(\mathbb{R}^n)}\to0$.
Since $p_c'<\infty$, a subsequence
$\widehat f_{j_k}$ converges to $\widehat f$ almost everywhere.

Let $w(\xi)=\abs{1-\abs\xi^2}^{-1}$ off $\mathbb{S}^{n-1}$ and set
$w=0$ on $\mathbb{S}^{n-1}$. This is a fixed, finite almost-everywhere,
nonnegative function. Fatou's lemma and
\eqref{eq:approximatebound} give
\begin{align*}
 \Q(f)
 &=\int_{\mathbb{R}^n} w(\xi)\abs{\widehat f(\xi)}^2\dd\xi\\
 &\leq\liminf_{k\to\infty}
       \int_{\mathbb{R}^n} w(\xi)\abs{\widehat f_{j_k}(\xi)}^2\dd\xi\\
 &\leq C_n\lim_{k\to\infty}\norm{f_{j_k}}_{L^{p_c}(\mathbb{R}^n)}^2
  =C_n\norm f_{L^{p_c}(\mathbb{R}^n)}^2.
\end{align*}
The last equality follows from strong $L^{p_c}$ convergence.
No integrable majorant for the singular weight is required by
Fatou's lemma. The Hilbert-valued integrals and logarithmic
identities were used only for Schwartz functions, before this
limiting step.
\end{proof}

\end{document}